\documentclass{icmart}
\usepackage{hyperref}
\usepackage{mathrsfs}
\usepackage{yhmath}

\contact[ardakov@maths.ox.ac.uk]{Mathematical Institute, University of Oxford, Oxford OX2 6GG}

\newtheorem*{thm}{Theorem}
\newtheorem*{defn}{Definition}

\theoremstyle{definition}

\newcommand{\Qp}{\mathbb{Q}_p}

\newcommand{\cD}[1]{\wideparen{\mathcal{D}}(#1)}

\renewcommand{\L}{\mathcal{L}}

\renewcommand{\O}{\mathcal{O}}

\newcommand{\T}{\mathcal{T}}

\newcommand{\hUK}[1]{\widehat{U(#1)_K}}
\newcommand{\hD}{\wideparen{\mathcal{D}}}

\newcommand{\invlim}{\lim\limits_{\longleftarrow}}
\newcommand{\qmb}[1]{\quad\mbox{#1}\quad}
  \let\leq=\leqslant
  \let\geq=\geqslant

\DeclareMathOperator{\Spec}{Spec}
\DeclareMathOperator{\Supp}{Supp}

\DeclareMathOperator{\GL}{GL}

\DeclareMathOperator{\Ext}{Ext}

\DeclareMathOperator{\Der}{Der}
\DeclareMathOperator{\RHom}{RHom}

\DeclareMathOperator{\an}{an}

\title[$\wideparen{\mathcal{D}}$-modules on rigid analytic spaces]{$\wideparen{\mathcal{D}}$-modules on rigid analytic spaces}

\author[Konstantin Ardakov]{Konstantin Ardakov}

\begin{document}

\begin{abstract}
We give an overview of the theory of $\wideparen{\mathcal{D}}$-modules on rigid analytic spaces and its applications to admissible locally analytic representations of $p$-adic Lie groups. 
\end{abstract}



\maketitle

\thispagestyle{empty}
\section{$\wideparen{\mathcal{D}}$-modules on rigid analytic spaces}

\subsection{Rigid analytic spaces} Let $K$ be a field complete with respect to a non-archimedean norm. The ultrametric inequality $|x-y| \leq \max |x|,|y|$ implies that the "unit circle" $\{|z|=1\}$ is \emph{open} in the affine line over $K$, and the "closed unit disc" $\{|z|\leq 1\}$ is disconnected, being the disjoint union of the unit circle and the open unit disc. This is a basic feature of non-archimedean geometry: every $K$-analytic manifold is totally disconnected with respect to its natural topology.

In order to make the category of $K$-analytic manifolds more geometric, Tate invented \emph{rigid analytic spaces}  \cite{Tate}, by introducing a new (Grothendieck) topology on this category, with a basis given by \emph{$K$-affinoid varieties} which are by definition the maximal ideal spectra of \emph{$K$-affinoid algebras}. The \emph{$n^{\rm{th}}$-Tate algebra} is the algebra $K \langle x_1,\ldots,x_n \rangle$ of $K$-valued functions on the $n$-dimensional polydisc that can be globally defined by a single power series which converges on the entire polydisc, and a $K$-affinoid algebra is by definition any homomorphic image of a Tate algebra. 

The theory of rigid analytic spaces has now reached maturity comparable to that of the theory of complex analytic manifolds, thanks to the works of Kiehl \cite{Kiehl}, Raynaud \cite{Raynaud}, Berkovich \cite{BerkBook}, Huber \cite{Huber} and many others. It is now an indispensable part of modern arithmetic geometry, and has found many striking applications such as Tate's uniformisation of elliptic curves with bad reduction, and the proof of the Local Langlands conjecture for $GL_n$ by Harris and Taylor.

\subsection{Rigid analytic quantisation}\label{RigQuant}
We assume now that $K$ is discretely valued, has characteristic zero and that its residue field has characteristic $p > 0$. Let $K^\circ$ denote its ring of integers and let $\pi \in K^\circ$ be a uniformiser. In a series of papers including \cite{Berth0}, \cite{Berth}, \cite{Berth2}, Berthelot introduced the sheaf of arithmetic differential operators $\widehat{\mathcal{D}_{\mathcal{X},\mathbb{Q}}^{(m)}}$ of level $m$ on every smooth formal $K^\circ$-scheme $\mathcal{X}$ in an attempt to better understand the $p$-adic cohomology of algebraic varieties in characteristic $p$. 

Let $\mathcal{X} = \widehat{\mathbb{A}^1}$ be the formal affine line over $K^\circ$. One of the origins of this work was the observation that it is possible to obtain the ring of global sections $\Gamma(\mathcal{X}, \widehat{\mathcal{D}_{\mathcal{X},\mathbb{Q}}^{(0)}})$ by defining a non-commutative multiplication $\ast$ on the second Tate algebra $K\langle x, y \rangle$, which is completely determined by the relation 

\vspace{-0.2cm}
\[y \ast x - x \ast y = 1.\]

In other words, if $t$ denotes a local coordinate on $\mathcal{X}$ and $\partial_t$ is the corresponding vector field then $\Gamma(\mathcal{X}, \widehat{\mathcal{D}_{\mathcal{X},\mathbb{Q}}^{(0)}})$ is isomorphic to the \emph{Tate-Weyl algebra} $K \langle t; \partial_t \rangle$ which can be defined by $\pi$-adically completing the usual Weyl algebra $A_1(K^\circ) = K^\circ[t;\partial_t]$ with coefficients in $K^\circ$, and inverting $\pi$. In this way we view  $\Gamma(\mathcal{X}, \widehat{\mathcal{D}_{\mathcal{X},\mathbb{Q}}^{(0)}})$ as a naive ``rigid analytic quantisation" of the two-dimensional polydisc. 

The aim of this paper is sketch the construction of the algebra $\cD{X}$ of infinite order differential operators on a smooth $K$-affinoid variety $X$, developed in joint work with S. J. Wadsley: proofs will appear elsewhere. Morally $\cD{X}$ is a   ``rigid analytic quantisation" of the \emph{entire} cotangent bundle $T^\ast X$.

\vspace{0.4cm}
\subsection{Lie algebroids} Let $k \to R$ be a morphism of commutative rings. Recall \cite{Rinehart} that a \emph{$k$-$R$-Lie algebra} or a \emph{Lie algebroid} is a pair $(L, a)$ consisting of a $k$-Lie algebra and $R$-module $L$, together with an $R$-linear $k$-Lie algebra homomorphism $a$ from $L$ to the set of $k$-linear derivations $\Der_kR$ of $R$, such that $[v, rw] = r[v,w] + a(v)(r)w$ for all $v,w \in L$ and all $r \in R$. It is possible to form a unital associative $k$-algebra $U(L)$ called the \emph{enveloping algebra} of $(L,a)$ which is generated as a $k$-algebra by $R$ and $L$, subject to appropriate natural relations. Enveloping algebras of Lie algebroids simultaneously generalise the ordinary enveloping algebra $U(\mathfrak{g})$ of a Lie algebra $\mathfrak{g}$ over a field $k$, and also the algebra $\mathcal{D}(X)$ of (crystalline) differential operators on a smooth affine algebraic variety $X$ over $k$, since $\mathcal{T}(X) = \Der_k \mathcal{O}(X)$ is itself naturally a $k$-$\mathcal{O}(X)$-Lie algebra such that $U(\mathcal{T}(X)) = \mathcal{D}(X)$. 

The ring $U(L)$ has a natural positive filtration with associated graded ring the symmetric $R$-algebra $\mathcal{S}(L)$ whenever $L$ is a projective $R$-module; thus $U(L)$ is an algebraic quantisation of the underlying topological space $\Spec \mathcal{S}(L)$. In this way, the enveloping algebra $U(\mathfrak{g})$ can be viewed as an algebraic quantisation of $\mathfrak{g}^\ast = \Spec \mathcal{S}(\mathfrak{g})$, and $\mathcal{D}(X)$ as an algebraic quantisation of the cotangent bundle $T^\ast X = \Spec \mathcal{S}(\mathcal{T}(X))$.

\vspace{0.4cm}
\subsection{Quantised rigid analytic cotangent bundles}\label{Dhat} Let $\mathcal{O}(X)$ be the algebra of rigid $K$-analytic functions on a smooth $K$-affinoid variety $X$, let $\mathcal{O}(X)^\circ$ be its subring of power-bounded elements, and let $\mathcal{T}(X)$ be the Lie algebra of continuous $K$-linear derivations of $\mathcal{O}(X)$. 

\begin{defn}
We say that an $\mathcal{O}(X)^\circ$-submodule $\L$ of $\mathcal{T}(X)$ is a \emph{Lie lattice} if it is a sub $K^\circ$-$\mathcal{O}(X)^\circ$-Lie algebra of $\mathcal{T}(X)$, is finitely generated as a module over $\mathcal{O}(X)^\circ$ and generates $\mathcal{T}(X)$ as a $K$-vector space. Let $\widehat{U(\L)}$ be the $\pi$-adic completion of $U(\L)$ and let $\hUK{\L} := \widehat{U(\L)} \otimes_{K^\circ} K$. We define $\cD{X}$ to be the inverse limit of the $\hUK{\L}$ where $\L$ runs over all possible Lie lattices in $\mathcal{T}(X)$.  
\end{defn}

\noindent Every Lie lattice $\L$ gives rise to a tower 
\[\hUK{\L} \leftarrow \hUK{\pi \L} \leftarrow \hUK{\pi^2 \L} \leftarrow \cdots\]
of Noetherian Banach $K$-algebras, whose inverse limit is a Fr\'echet $K$-algebra in the sense of non-archimedean functional analysis \cite{ST}. Since any two Lie lattices in $\T(X)$ contain a $\pi$-power multiple of each other, the inverse limit of this tower is isomorphic to $\cD{X}$, regardless of the choice of the Lie lattice $\L$. 

\vspace{0.4cm}
\noindent \textbf{Example.} Let $X$ denote the closed disc of radius $1$ in the affine line over $K$, with local coordinate $t$. Then $\mathcal{O}(X)$ is the first Tate algebra $K\langle t \rangle$ and $\mathcal{O}(X)^\circ$ is the subalgebra $K^\circ\langle t \rangle := K\langle t \rangle \cap K^\circ[[t]]$. Let $\L = K^\circ \langle t \rangle \partial_t$, so that $\pi^n \L$ is a Lie lattice in $\T(X)$ for each $n \geq 0$ and 
\[ \hUK{\pi^n \L} \cong K \langle t; \pi^n \partial_t\rangle\]
is a deformation of the Tate-Weyl algebra over $K$. Thus 
\[\cD{X} = \bigcap_{n \geq 0} K \langle t; \pi^n \partial_t\rangle =  
\left\{ \sum_{i=0}^\infty a_i \partial_t^i \in K\langle t\rangle [[\partial_t]] : \lim\limits_{i\to\infty} \frac{a_i}{\pi^{in}} = 0 \qmb{for all} n \geq 0\right\}\]
is naturally in bijection with $\O(T^\ast X)$.
\vspace{0.4cm}

If $Y \hookrightarrow X$ is an open embedding of smooth $K$-affinoid varieties and $\L$ is a Lie lattice in $\T(X)$, then $\mathcal{O}(Y)^\circ \otimes_{\mathcal{O}(X)^\circ} \L$ need not be a Lie lattice in $\T(Y)$ in general. However,  a sufficiently large $\pi$-power multiple of $\mathcal{O}(Y)^\circ \otimes_{\mathcal{O}(X)^\circ} \L$ \emph{is} a Lie lattice in $\T(Y)$, and the functoriality of enveloping algebras of Lie algebroids induces a ring map $\cD{X} \to \cD{Y}$. We have the following non-commutative analogue of Tate's Acyclicity Theorem:

\begin{thm} Let $X$ be a smooth $K$-affinoid variety. Then $\hD$ is a sheaf on $X$ with vanishing higher cohomology.
\end{thm}

This construction extends naturally to a sheaf of $K$-algebras $\hD$ on arbitrary smooth rigid analytic varieties over $K$.
\subsection{Coadmissible $\hD$-modules} Recall \cite{ST3} that Schneider and Teitelbaum defined a \emph{Fr\'echet-Stein} algebra to be the inverse limit of a countable inverse system of Noetherian $K$-Banach algebras $(A_n)_{n\in\mathbb{N}}$ with flat transition maps. 

\begin{thm} Let $X$ be a smooth $K$-affinoid variety. Then the algebra $\cD{X}$ is Fr\'echet-Stein.
\end{thm}
There is a well-behaved abelian category of \emph{coadmissible $A$-modules} associated with any Fr\'echet-Stein algebra $A$, whose objects are inverse limits of compatible familes $(M_n)_{n\in\mathbb{N}}$ where each $M_n$ is a finitely generated module over $A_n$. Let $X$ be a smooth rigid $K$-analytic variety, and let $(X_j)_j$ be a sufficiently fine admissible $K$-affinoid covering of $X$. It is possible to prove a precise non-commutative analogue of Kiehl's Theorem from \cite{Kiehl}, which allows us to glue the resulting categories of coadmissible $\cD{X_j}$-modules in an appropriate way in order to obtain the category $\mathcal{C}_X$ of \emph{coadmissible $\hD$-modules} on $X$.

Every $\hD$-module that is coherent as an $\mathcal{O}_X$-module is coadmissible in this sense. As in the classical theory \cite{HTT} over $\mathbb{C}$, we may think of these $\hD$-modules as rigid vector bundles equipped with a flat connection, and thereby obtain a link between our $\hD$-modules and the well-established theory of $p$-adic differential equations \cite{Kedlaya}. There is also a natural exact analytification functor from the category of coherent $\mathcal{D}$-modules on a smooth algebraic variety $Y$ over $K$ to $\mathcal{C}_{Y^{an}}$. For these reasons, we will regard $\mathcal{C}_X$ as an appropriate rigid analytic analogue of the category of coherent algebraic $\mathcal{D}$-modules. 

\subsection{Functoriality}
In the classical setting \cite{HTT}, it is known that the inverse and direct image functors for $\mathcal{D}$-modules preserve $\mathcal{O}$-quasi-coherence, but need not in general preserve $\mathcal{D}$-coherence. Since our category of coadmissible $\hD$-modules is modelled on the category of \emph{coherent} algebraic $\mathcal{D}$-modules, and since it is well-known that there is no obvious well-behaved analogue of quasi-coherent $\mathcal{O}$-modules in rigid analytic geometry, it is unreasonable to expect to be able to define direct and inverse image functors in full generality in our current  setting.  However, given a morphism $f : X \to Y$ between smooth rigid analytic varieties, it is possible to define a \emph{transfer bimodule} $\hD_{X \to Y} := \mathcal{O}_X \wideparen{\otimes}_{f^{-1} \mathcal{O}_Y} f^{-1} \hD_Y$, and a direct image functor
\[\begin{array}{ccccc} f_+ &:& \mathcal{C}^r_X &\to& \mathcal{C}^r_Y  \\ & & \mathcal{M} & \mapsto & f_\ast\left(\mathcal{M} \wideparen{\otimes}_{\hD_X} \hD_{X \to Y} \right)\end{array}\]
between the derived categories of coadmissible right $\hD$-modules, at least in the case when $f$ is \emph{proper}. It would be interesting to investigate whether the classical inverse and direct image functors extend to our setting in a greater generality.

\subsection{Dimension theory}\label{DimTh}
Whenever $A$ is an \emph{Auslander-regular ring} \cite{Clark}, the functor $M \mapsto \RHom_A(M, A)$ induces  an anti-equivalence between the derived categories of finitely generated left, and right, $A$-modules \cite{YZ}. This us allows to associate with any finitely generated $A$-module $M$ its \emph{canonical dimension} $d(M)$, defined in terms of the vanishing of the $\Ext$ groups $\Ext_A^j(M,A)$. 

When $A$ is the ring of regular functions on a smooth affine variety $X$ over a field, $d(M)$ is the Krull dimension of the support of the associated sheaf $\widetilde{M}$ on $X$.  

\begin{thm} Let $X$ be the $d$-dimensional polydisc and let $\L$ be the free $\O(X)^\circ$-submodule of $\T(X)$ spanned by the standard vector fields. For every $n \geq 0$, the deformed Tate-Weyl algebra $\hUK{\pi^n \L}$ is an Auslander-regular ring of global dimension $d$.
\end{thm}
\begin{proof} This version of Bernstein's Inequality for deformed Tate-Weyl algebras follows from \cite[Theorem B]{AW13}.\end{proof}
Schneider and Teitelbaum observed in \cite[\S 8]{ST} that if $A = \invlim A_n$ is a Fr\'echet-Stein algebra such that each $A_n$ is Auslander-regular of the same global dimension, then the canonical dimension function extends naturally to the category of coadmissible $A$-modules. It follows from the above result that their formalism applies to our algebras $\cD{X}$ whenever $X$ is sufficiently small, and allows us to define the canonical dimension of a coadmissible $\hD$-module on an arbitrary smooth rigid $K$-analytic variety.  

\begin{defn}
We say that a non-zero coadmissible $\hD$-module is \emph{holonomic} if its canonical dimension is zero.
\end{defn}

\subsection{Support and Kashiwara's equivalence}\label{Kash}

The support of an abelian sheaf on a topological space is a fundamental invariant. Since our sheaves are defined on a space with a Grothendieck topology, the usual definition of support in terms of stalks seems inferior to the alternative one given by
\[ \Supp \mathcal{M} := X - \bigcup \left\{ U \hspace{0.3cm} \mbox{admissible open in } X: \mathcal{M}_{|U} = 0\right\}.\]
It is natural to hope that $\Supp \mathcal{M}$ is an analytic subspace of $X$ for every coadmissible $\hD$-module $\mathcal{M}$. However, morally a coadmissible $\hD$-module is a coherent sheaf on a rigid analytic quantisation of $T^\ast X$ and the projection map $T^\ast X \to X$ isn't proper, so this hope is probably unreasonable. Nevertheless, it seems possible that there is an appropriately large subcategory of coadmissible $\hD$-modules whose objects do have analytic support.

As there is no natural exhaustive ring filtration on the sheaf $\hD$ due to the presence of completions, it is not clear at present how to define a good analogue of the \emph{characteristic variety} for coadmissible $\hD$-modules.  Nevertheless it is conceivable that in the future it will be possible to do this by ``microlocalising" coadmissible $\hD$-modules to appropriate Lagrangian affinoid subspaces of $T^\ast X$, and thereby make more precise the words ``rigid analytic quantisation". In any case, the notion of support defined above is sufficient for us to be able to formulate a rigid-analytic version of the fundamental \emph{Kashiwara equivalence}:

\begin{thm} Let $i : Y \to X$ be a closed immersion of smooth rigid analytic varieties. Then the functor $i_+$ induces an equivalence of abelian categories between $\mathcal{C}_Y$ and the full subcategory $\mathcal{C}_X^Y$ consisting of objects $\mathcal{M}$ in $\mathcal{C}_X$ with support contained in the image of $Y$.
\end{thm}

\section{$p$-adic representations of $p$-adic Lie groups}

\subsection{Locally analytic representations}
Let $L$ be a finite extension of $\Qp$, assume that our ground field $K$ contains $L$ and let $G$ be a locally $L$-analytic group. In a series of papers including \cite{ST1}, \cite{ST0}, \cite{ST}, \cite{ST3}, Schneider and Teitelbaum developed the theory of \emph{admissible locally analytic $G$-representations} in locally convex $K$-vector spaces. This theory has found applications to several areas, including $p$-adic automorphic forms \cite{Loeffler}, $p$-adic interpolation \cite{EmertonInter}, non-commutative Iwasawa theory \cite{ST2} and the $p$-adic local Langlands programme \cite{Berger}, \cite{Colmez}, \cite{BreuilICM}, \cite{BreuilLocAnSocle}.

By definition, the \emph{locally analytic distribution algebra of $G$ over $K$} is the strong dual $D(G,K)$ of the vector space of locally analytic $K$-valued functions on $G$. It may be viewed as a certain $K$-Fr\'echet space completion of the group ring $K[G]$.

When the group $G$ is compact, Schneider and Teitelbaum showed that $D(G,K)$ is a Fr\'echet-Stein algebra, so the notion of coadmissible $D(G,K)$-module makes sense. A locally analytic representation $V$ of an arbitrary locally $L$-analytic group $G$ is \emph{admissible} if its strong dual is coadmissible as a module over the distribution algebra $D(H,K)$ of every compact open subgroup $H$ of $G$.

One of the most basic problems in this theory is to gain a better understanding of the \emph{irreducible} admissible locally analytic representations of $G$, or equivalently, the simple coadmissible modules over the distribution algebra $D(G,K)$.  

\vspace{0.2cm}
\subsection{Arens-Michael envelopes}
There is a natural embedding of the Lie algebra $\mathfrak{g}$ of $G$ into $D(G,K)$, which extends to an embedding of $K$-algebras $U(\mathfrak{g}_K) \hookrightarrow D(G,K)$, where $\mathfrak{g}_K := K \otimes_L \mathfrak{g}$. It follows from the work of Kohlhaase \cite{Kohlhaase} that the closure of the image consists of the $K$-valued locally analytic distributions on $G$ which are supported at the identity in a suitable sense, and is isomorphic to the Hausdorff completion $\wideparen{U(\mathfrak{g}_K)}$ of $U(\mathfrak{g}_K)$ with respect to all submultiplicative seminorms on $U(\mathfrak{g}_K)$. Following Schmidt \cite{SchmidtArens}, we call this completion the \emph{Arens-Michael envelope} of $U(\mathfrak{g}_K)$.

If $\{x_1,\ldots,x_d\}$ is a $K$-basis for $\mathfrak{g}_K$, then $\wideparen{U(\mathfrak{g}_K)}$ can be identified with the vector space of power series in the $x_i$ converging everywhere on $K^d$:
\[\wideparen{U(\mathfrak{g}_K)} = \left\{ \sum_{\alpha\in\mathbb{N}^d} \lambda_\alpha \mathbf{x}^\alpha \in K[[x_1,\ldots,x_d]] : \sup_{\alpha\in\mathbb{N}^d} |\lambda_\alpha|r^{-|\alpha|} < \infty \qmb{for all} r> 0\right\}.\]
This allows us to view $\wideparen{U(\mathfrak{g}_K)}$ as a ``rigid analytic quantisation" of $\mathfrak{g}^\ast_K$.

\vspace{0.2cm}
\subsection{Infinitesimal central characters}\label{InfChars}
Assume from now on that $G$ is an open subgroup of the group of $L$-rational points of a split semisimple $L$-algebraic group $\mathbf{G}$. Let $\mathfrak{g}$ be the Lie algebra of $G$. The classical ``Harish-Chandra" centre $Z(\mathfrak{g}_K)$ of $U(\mathfrak{g}_K)$ remains central in $D(G,K)$, and Kohlhaase showed that the Arens-Michael envelope of $Z(\mathfrak{g}_K)$ is in fact the centre of $D(G,K)$ whenever the centre of $G$ is trivial.
\vspace{0.2cm}
\begin{thm} Let $M$ be a simple coadmissible $\wideparen{U(\mathfrak{g}_K)}$-module. Then there exists a $K$-algebra homomorphism $\theta_M : Z(\mathfrak{g}_K) \to \overline{K}$ such that $z \cdot m = \theta_M(z)m$ for all $z \in Z(\mathfrak{g}_K)$ and $m\in M$. Thus $M$ has an \emph{infinitesimal central character}.
\end{thm}

This result follows from our analogue of Quillen's Lemma \cite[Theorem D]{AW13} for affinoid enveloping algebras. Dospinescu and Schraen have extended this Theorem to simple coadmissible $D(G,K)$-modules in \cite{DoSch}.

\subsection{Beilinson-Bernstein Localisation}\label{BBLoc}
It follows from Theorem \ref{InfChars} that in the quest for simple coadmissible $\wideparen{U(\mathfrak{g}_K)}$-modules, it will be sufficient to study the central quotients
\[ \wideparen{\mathcal{U}^\theta} := \wideparen{U(\mathfrak{g}_K)} / \langle \ker \theta\rangle\]
for every central character $\theta :  Z(\mathfrak{g}_K) \to \overline{K}$ in turn. It is well-known that a good way to understand the uncompleted algebras $U(\mathfrak{g}_K) / \langle \ker \theta\rangle$ is through \emph{geometric representation theory} \cite{BB}, which interprets them as rings of globally defined twisted differential operators on the flag variety $\mathbf{G} / \mathbf{B}$ associated with $\mathfrak{g}_K$. 

\begin{thm} Let $(\mathbf{G} / \mathbf{B})^{\an}$ be the rigid analytic flag variety. Let $\mathfrak{t}_K$ be a Cartan subalgebra of $\mathfrak{g}$ and let $\lambda \in \mathfrak{t}_K^\ast$ be a dominant regular weight. Then there is an equivalence of abelian categories 
\[
\left\{ 
				\begin{array}{c} 
					coadmissible \\ 
					\wideparen{\mathcal{U}^{\lambda\phi}}\hspace{-0.1cm}-\hspace{-0.1cm}modules
				\end{array}
\right\} \cong \left\{
				\begin{array}{c}
				 coadmissible \hspace{0.1cm} \\ 
				  \wideparen{\mathcal{D}^\lambda}\hspace{-0.1cm}-\hspace{-0.1cm}modules \hspace{0.1cm} on \hspace{0.1cm} (\mathbf{G}/\mathbf{B})^{\an} 
				\end{array}
\right\}
\]
where $\phi : Z(\mathfrak{g}_K) \to S(\mathfrak{t}_K)$ is the Harish-Chandra homomorphism. \end{thm}

Here $\wideparen{\mathcal{D}^\lambda}$ denotes a $\lambda$-twisted version of the ring $\wideparen{\mathcal{D}}$ from $\S \ref{Dhat}$. This rigid analytic analogue of the Beilinson-Bernstein Localisation Theorem has several precursors, including \cite[Theorem 3.2]{BMR1},  \cite[Th\'eor\`eme 2.1]{Noot1} and \cite[Theorem C]{AW13}.  

\subsection{Canonical dimension estimates}\label{DimEst}

Schneider and Teitelbaum's dimension theory from \cite[\S 8]{ST} applies not only to our algebras $\cD{X}$ as explained in $\S \ref{DimTh}$ above, but also to the Arens-Michael envelopes $\wideparen{U(\mathfrak{g}_K)}$ and the distribution algebras $D(G,K)$ whenever $G$ is compact locally $\Qp$-analytic group. The canonical dimension of a coadmissible $D(G,K)$-module $M$ is zero precisely when $M$ is finite dimensional as a $K$-vector space. 

Using the folklore observation \cite{BezrukavnikovICM} that the main mechanism behind the Beilinson-Bernstein Localisation Theorem is a quantisation of the Springer resolution,  we obtain the following analogue of Bernstein's Inequality for $\wideparen{U(\mathfrak{g}_K)}$.
\begin{thm} Suppose that $p$ is a very good prime for $\mathbf{G}$. Let $r$ be the half the smallest possible dimension of a non-zero $\mathbf{G}(K)$-orbit in $\mathfrak{g}_K^\ast$ and let $M$ be a coadmissible $\wideparen{U(\mathfrak{g}_K)}$-module. Then either $d(M) = 0$ or $d(M) \geq r$.
\end{thm}

We refer the reader to \cite[\S 6.8, \S 9.9]{AW13} for the meaning of the words ``very good prime", and the precise values that the invariant $r$ takes. Roughly speaking, $r$ is the square root of the dimension of $G$: for example if $\mathbf{G} = SL_n$ then $r = n - 1$. Theorem \ref{DimEst} is an analogue of Smith's Theorem for classical enveloping algebras of complex semisimple Lie algebras \cite{Smith}, and follows easily from the corresponding statement for semisimple affinoid enveloping algebras \cite[Theorem 9.10]{AW13}.  A similar estimate holds for semisimple Iwasawa algebras \cite[Theorem A]{AW13}, and semisimple locally analytic distribution algebras \cite[Theorem 9.9]{SchmidtDistLoc}.

\subsection{Equivariant $\hD$-modules}
At the time of writing, the main applications of our methods to the theory of locally analytic representations have been the dimension estimates explained above. However, we believe that there is significant scope for other applications. Using $\hD$-modules it should be possible to construct irreducible coadmissible $D(G,K)$-modules geometrically, and to better understand the admissible representations arising in the $p$-adic local Langlands programme for $\GL_2(\Qp)$ and for other $p$-adic Lie groups.

There have been several attempts to prove a version of the Beilinson-Bernstein Localisation Theorem for locally analytic distribution algebras, including \cite{SchmidtDistLoc} and \cite{PSS}. We expect that it will be possible in the future to show that the abelian category of admissible locally analytic representations of $G$ with dominant regular infinitesimal central character $\lambda \phi$ is anti-equivalent to the category of \emph{coadmissible $G$-equivariant $\wideparen{\mathcal{D}^\lambda}$-modules} on the rigid analytic flag variety.

\end{document}